\documentclass[12pt,reqno]{amsart}

\usepackage{amsthm,amssymb,amstext,amscd,euscript,mathrsfs,amsfonts,amsbsy,amsxtra,latexsym,amsmath}
\usepackage{fullpage}
\usepackage[english]{babel}
\usepackage[latin1]{inputenc}
\usepackage{verbatim}
\usepackage{graphicx} 
\usepackage[all]{xy} 
\usepackage{enumitem}
\usepackage{bbm}
\numberwithin{figure}{section} 
\allowdisplaybreaks

\numberwithin{equation}{section}
\newtheorem{theorem}{\textbf{Theorem}}
\numberwithin{theorem}{section}

\newtheorem{lemma}[theorem]{\textbf{Lemma}}

\theoremstyle{definition}
\newtheorem{definition}[theorem]{Definition}
\newtheorem{remark}{Remark}[section]

\newcommand{\bea}{\begin{eqnarray}} 
\newcommand{\eea}{\end{eqnarray}} 
\newcommand{\be}{\begin{equation}} 
\newcommand{\ee}{\end{equation}} 
\newcommand{\benn}{\begin{equation*}} 
\newcommand{\eenn}{\end{equation*}}

\title[short title]{On a relative form of Verdier specialization}
\author{James Fullwood and Dongxu Wang}
\address{Department of Mathematics\\ University of Hong Kong\\ Pokfulam Road, Hong Kong.}
\email{fullwood@maths.hku.hk}
\address{Department of Mathematics\\ Dongbei University of Finance and Economics\\ 217 Jianshan St, Shahekou District, Dalian, Liaoning, China.}
\email{dxwang1981@gmail.com}
\begin{document}

\maketitle
\begin{center}
\emph{Dedicated to Henry Laufer on his 70th birthday}
\end{center}

\begin{abstract}
We prove a relative form of Verdier's specialization formula, and apply it to derive a Chern class identity predicted by string dualities. 
\end{abstract}

\section{Introduction}
Let $\mathscr{Y}\to \Delta$ be a family over a disk about the origin in $\mathbb{C}$ such that the total space of the family is a topological locally trivial fibration over $\Delta \setminus 0$. Denote a fiber over $t\neq 0$ by $Y_t$ and the central fiber by $Y_0$. In \cite{VS}, Verdier defines specialization morphisms
\[
\sigma_H:H_*(Y_t)\to H_*(Y_0), \quad \quad \sigma_F:F(\mathscr{Y})\to F(Y_0), 
\]
where $H_*$ denotes the integral homology functor and $F$ denotes the functor that takes a variety to its group of constructible functions. These morphisms are functorial in the sense that they commute with taking Chern classes, as Verdier proves that for any constructible function $\vartheta \in F(\mathscr{Y})$ and $t$ sufficiently small
\begin{equation}\label{VF}
\sigma_Hc_{\text{SM}}(\left.\vartheta\right|_{Y_t})=c_{\text{SM}}(\sigma_F(\vartheta)),
\end{equation}
where $c_{\text{SM}}(\cdot)$ denotes the morphism which takes a constructible function to its \emph{Chern-Schwartz-MacPherson class}. Chern-Schwartz-MacPherson classes evaluated at the constant function 1 yield a generalization of the total homological Chern class to possibly singular varieties which extend the Gau{\ss}-Bonnet-Chern theorem to the singular setting, since for $X$ possibly singular we have
\[
\int_X c_{\text{SM}}(\mathbbm{1}_X)=\chi(X),
\] 
where $\mathbbm{1}_X$ denotes the indicator function of $X$ (i.e., the constant function 1) and $\chi(X)$ its topological Euler characteristic with compact support. For the indicator function $\mathbbm{1}_{\mathscr{Y}}$ of the total space of the family $\mathscr{Y}\to \Delta$, Verdier's specialization $\sigma_F \mathbbm{1}_{\mathscr{Y}}\in F(Y_0)$ to the central fiber is such that 
\[
\int_{Y_0}c_{\text{SM}}(\sigma_F \mathbbm{1}_{\mathscr{Y}})=\chi(Y_t),
\]
and thus yields a deformation invariant generalization of Euler characteristic to the realm of singular varieties. In the case that the central fiber $Y_0$ is regularly embedded the class $c_{\text{SM}}(\sigma_F \mathbbm{1}_{\mathscr{Y}})\in H_*Y_0$ coincides with a characteristic class for singular varieties and schemes referred to as the \emph{Chern-Fulton class} (\cite{IT}, \S~4.2.3), in which case we have
\[
c_{\text{SM}}(\sigma_F \mathbbm{1}_{\mathscr{Y}})=c(T_{\text{vir}}Y_0)\cap [Y_0]\in H_*Y_0,
\]
where $T_{\text{vir}}Y_0$ denotes the virtual tangent bundle of $Y_0$, i.e., $\left.TM\right|_{Y_0}-N_{Y_0}M$. In this note, we fix a smooth complex base variety $B$ and consider the case where all fibers of the family $\mathscr{Y}\to \Delta$ come equipped with a morphism $\varphi_t:Y_t\to B$, and prove a relative version of formula (\ref{VF}) for the case $\vartheta=\mathbbm{1}_{\mathscr{Y}}$. In particular, we fix a proper morphism $f:X\to B$ of smooth complex varieties, a vector bundle $\mathscr{E}\to X$, and let 
\[
s:X\times \Delta\to \mathscr{E}
\]
be a regular morphism such that $s_t=s(\cdot,t)$ is a generic section of $\mathscr{E}\to X$ for all $t\in \Delta\setminus 0$, with $s_0$ possibly non-generic. Considering the zero-schemes of the sections $s_t$ as $t$ varies in $\Delta$ gives rise to a family $\mathscr{Y}\to \Delta$, such that each fiber is a closed subscheme of $X$ regularly embedded by the natural inclusion $Y_t\overset{\iota_t}\hookrightarrow X$, whose normal bundle is $\iota_t^*\mathscr{E}$. By composing the natural inclusion associated with each fiber with the morphism $f$ yields proper morphisms 
\[
\varphi_t=f\circ \iota_t:Y_t\to B
\]  
for all $t\in \Delta$. Our result is then given by the following
\begin{theorem}\label{mt}
With notations as given above
\begin{equation}\label{mf}
c_{\emph{SM}}\left(\left.{\varphi_t}_*\mathbbm{1}_{\mathscr{Y}}\right|_{Y_t}\right)=c_{\emph{SM}}\left({\varphi_0}_*\sigma_F \mathbbm{1}_{\mathscr{Y}}\right)\in H_*B,
\end{equation}
where ${\varphi_t}_*$ and ${\varphi_0}_*$ denote the proper pushforward of constructible functions associated with the morphisms $\varphi_t$ and $\varphi_0$ respectively.
\end{theorem}

Explicitly computing both sides of equation (\ref{mf}) for a given family $\mathscr{Y}\to \Delta$ subject to the assumptions above then yields a non-trival Chern class identity between seemingly unrelated varieties, as the constructible functions $\left.{\varphi_t}_*\mathbbm{1}_{\mathscr{Y}}\right|_{Y_t}$ and ${\varphi_0}_*\sigma_F \mathbbm{1}_{\mathscr{Y}}$ will in general be quite different. We were motivated to derive such a formula in our efforts to yield a purely mathematical explanation for the existence of certain Chern class identities predicted by string dualities in physics. In particular, a regime of string theory referred to as `F-theory' models the purported compactified dimensions of spacetime by an elliptic fourfold $Y\to B$ (with $B$ a Fano threefold), and the preservation of a quantity referred to as `D3 charge' under S-duality with a weakly coupled type-IIB orientifold theory predicts a relation of the form
\[
\chi(Y)=\sum_i \chi(D_i),
\] 
where $D_i$ are hypersurfaces in $B$ which support D7-branes in the type-IIB theory. Such relations --referred to as tadpole relations-- provide a consistency check between the two theories, as well as yield non-trivial mathematical identities among Euler characteristics of the varieties $Y$ and the $D_i$. Moreover, it has been shown that these tadpole relations coming from physics are degree-zero terms of Chern class identities which hold in a much broader context than their physical origins \cite{AE1}\cite{AE2}\cite{EFY}\cite{EJY}. In particular, the Chern class identities which yield the tadpole relations in physics hold for elliptic fibrations $\varphi:Y\to B$ with $B$ of arbitrary dimension and $Y$ not necessarily Calabi-Yau. In \cite{TR2}, it was shown by deriving a special case of Theorem~\ref{mt} that the Chern class identities predicted by physics are manifestations of a relative form of Verdier specialization, thus providing a purely mathematical explanation for the existence of such identities.   

In what follows, we review the theory of Chern-Schwartz-MacPherson classes of constructible functions, prove Theorem~\ref{mt}, and then use Theorem~\ref{mt} to derive a non-trivial Chern class identity predicted by string dualities. For those acquainted with the physical significance of such identities, we close with a proposed general definition of `orientifold Euler characteristic', which captures the contribution of a singular brane to the D3 charge.   

\section{Chern-Schwartz-MacPherson classes and a proof of Theorem~\ref{mt}}
Let $X$ be a complex variety. A constructible function on $X$ is an integer-valued function of the form
\[
\sum_i a_i\mathbbm{1}_{W_i},
\]
with each $a_i\in \mathbb{Z}$, $W_i\subset X$ a closed subvariety and $\mathbbm{1}_{W_i}$ the function that evaluates to $1$ for points inside of $W_i$ and is zero elsewhere. The collection of all such functions forms an abelian group under addition, and is referred to as the \emph{group of constructible functions} on $X$, denoted $F(X)$. A proper morphism $f:X\to Y$ induces a functorial group homomorphism $f_*:F(X)\to F(Y)$, which by linearity is determined by the prescription
\begin{equation}\label{cfd}
f_*\mathbbm{1}_W(p)=\chi\left(f^{-1}(p)\cap W\right),
\end{equation}
where $W\subset X$ is a closed subvariety and $\chi$ denotes topological Euler characteristic with compact support. By taking $F(f)=f_*$, we may view $F$ as a covariant functor from varieties to abelian groups. Now denote by $H_*$ homology functor, which takes a variety to its integral homology. Both the the constructible function functor and the homology functor are covariant with respect to proper maps. In the 1960s Deligne and Grothendieck conjectured the existence of a unique natural transformation
\[
c_*:F\to H_*
\]
such that for $X$ smooth
\[
c_*(\mathbbm{1}_X)=c(TX)\cap [X]\in H_*X,
\]
i.e., the total homological Chern class of $X$. For $X$ possibly singular the class $c_*(\mathbbm{1}_X)$ would then be a generalization of Chern class to the realm of singular varieties. Moreover, functoriality would then imply
\begin{equation}\label{ecf}
\int_X c_*(\mathbbm{1}_X)=\chi(X),
\end{equation}
so that such a class would provide a natural generalization of the Gau{\ss}-Bonnet-Chern theorem to the singular setting. In 1974, Deligne and Grothendieck's conjecture was proved by Robert MacPherson \cite{RMCC}, and as such the class $c_*(\mathbbm{1}_X)$ became known as MacPherson's Chern class. It was shown later by Brasselet and Schwartz that MacPherson's Chern class was in fact the Alexander dual in relative cohomology of singular Chern classes constructed by Marie-H\'{e}l\`{e}ne Schwartz in the 1960s using radial vector fields \cite{BSCC}\cite{MHSCC}, thus the class $c_*(\mathbbm{1}_X)$ is referred to at present as the \emph{Chern-Schwartz-MacPherson class}, and is denoted by $c_{\text{SM}}(X)$. The class $c_*(\delta)$ of a general constructible function $\delta$ will be denoted from here on by $c_{\text{SM}}(\delta)$.

The proof of Theorem~\ref{mt} essentially follows from the functoriality of Chern-Schwartz-MacPherson classes and elementary manipulations, but before doing so we recall our assumptions. So let $B$ be a smooth complex variety, $f:X\to B$ be a proper morphism with $X$ smooth, and let $\mathscr{E}\to X$ be a vector bundle. Denote by $\Delta$ a disk about the origin in $\mathbb{C}$. We assume 
\[
s:X\times \Delta\to \mathscr{E}
\]   
is a regular morphism such that $s_t=s(\cdot,t)$ is a generic section of $\mathscr{E}\to X$ for $t\in \Delta\setminus 0$, with $s_0$ a possibly non-generic section. The zero-schemes of $s_t$ as $t$ varies in $\Delta$ then gives rise to a family $\mathscr{Y}\to \Delta$, such that the fibers $Y_t$ come equipped with morphisms $\varphi_t:Y_t\to B$, where $\varphi_t=f\circ \iota_t$ with $\iota_t:Y_t\hookrightarrow X$ the natural inclusion. In order to apply Verdier's specialization morhisms we assume that the family is a topological locally trivial fibration over $\Delta\setminus 0$. We will hold off giving a precise definition of Verdier's specialization until \S\ref{s3} as it is not needed for the proof of Theorem~\ref{mt} modulo the acceptance of formula (\ref{VF}).

\begin{proof}[Proof of Theorem~\ref{mt}] 
We recall that our goal is to show
\[
c_{\text{SM}}\left(\left.{\varphi_t}_*\mathbbm{1}_{\mathscr{Y}}\right|_{Y_t}\right)=c_{\text{SM}}\left({\varphi_0}_*\sigma_F \mathbbm{1}_{\mathscr{Y}}\right).
\]
For this, we first note that it follows from the proof of Theorem~5.3 in \cite{PPCC} that for $t\neq 0$ we have
\[
\sigma_Hc_{\text{SM}}(Y_t)=\frac{\iota_0^*c(TX)}{\iota_0^*c(\mathscr{E})}\cap [Y_0],
\]
thus by formula (\ref{VF}) we have 
\begin{eqnarray*}
c_{\text{SM}}(\sigma_F\mathbbm{1}_{\mathscr{Y}})&=&\sigma_Hc_{\text{SM}}(\left.\mathbbm{1}_{\mathscr{Y}}\right|_{Y_t}) \\
                                                &=&\sigma_Hc_{\text{SM}}(\mathbbm{1}_{Y_t}) \\
																								&=&\sigma_Hc_{\text{SM}}(Y_t) \\
																								&=&\frac{\iota_0^*c(TX)}{\iota_0^*c(\mathscr{E})}\cap [Y_0].
\end{eqnarray*}
We then have
\begin{eqnarray*}
c_{\text{SM}}\left({\varphi_0}_*\sigma_F \mathbbm{1}_{\mathscr{Y}}\right)
&=&{\varphi_0}_*c_{\text{SM}}(\sigma_F\mathbbm{1}_{\mathscr{Y}}) \\
&=&{\varphi_0}_*\left(\frac{\iota_0^*c(TX)}{\iota_0^*c(\mathscr{E})}\cap [Y_0]\right) \\
&=&f_*\circ {\iota_0}_*\left(\frac{\iota_0^*c(TX)}{\iota_0^*c(\mathscr{E})}\cap [Y_0]\right) \\
&=&f_*\left(\frac{c(TX)}{c(\mathscr{E})}\cap {\iota_0}_*[Y_0]\right) \\
&=&f_*\left(\frac{c(TX)}{c(\mathscr{E})}\cap {\iota_t}_*[Y_t]\right) \\
&=&f_*\circ {\iota_t}_*\left(\frac{\iota_t^*c(TX)}{\iota_t^*c(\mathscr{E})}\cap [Y_t]\right) \\
&=&{\varphi_t}_*c_{\text{SM}}(\mathbbm{1}_{Y_t}) \\
&=&{\varphi_t}_*c_{\text{SM}}(\left.\mathbbm{1}_{\mathscr{Y}}\right|_{Y_t}) \\
&=&c_{\text{SM}}\left({\varphi_t}_*\left.\mathbbm{1}_{\mathscr{Y}}\right|_{Y_t}\right), \\
\end{eqnarray*}
as desired, where in the first and final equalities we use functoriality of $c_{\text{SM}}$, the second equality follows from our first chain of equalities, the third and seventh by functoriality of proper pushforward, the fourth and sixth by the projection formula, and the fifth since $Y_t$ and $Y_0$ are both sections of $\mathscr{E}\to X$. 
\end{proof}

\section{Derivation of a Chern class identity predicted by F-theory/type-IIB duality}\label{s3}
The geometric apparatus of a regime of string theory referred to as F-theory is an elliptic fibration over a Fano threefold $\varphi:Y\to B$, whose total space $Y$ --which plays the r\^{o}le of the compactified dimensions of spacetime-- is a Calabi-Yau fourfold. A certain form of S-duality then identifies F-theory with a weakly coupled orientifold type-IIB theory, which is realized geometrically by a singular degeneration (satisfying certain conditions coming from physics), not only of the total space $Y$, but of the \emph{fibration} $\varphi:Y\to B$, to a fibration $\varphi_0:Y_0\to B$ where all the fibers are singular degenerations of elliptic curves. The duality is then captured by the corresponding family over a disk
\[
\mathscr{Y}\to \Delta,
\] 
such that $Y=Y_{t_0}$ for some $t_0\in \Delta$ and each fiber $Y_t$ comes equipped with a morphism $\varphi_t:Y_t\to B$ for all $t\in \Delta$. Such a family is often referred to as a `weak coupling limit' of F-theory. There are certain divisors $D_i$ in $B$ associated with the central fiber of the family $Y_0$ whose Euler characteristics represent in the type-IIB theory what physicists refer to as `D3 charge'. Since on the F-theory side the total D3 charge is given by the Euler characteristic of $Y$, the preservation of D3 charge under S-duality then leads one who is confident in the theory to predict the relation
\begin{equation}\label{tr1}
\chi(Y)\overset{?}=\sum_i \chi(D_i),
\end{equation}   
which should necessarily hold if indeed F-theory and type-IIB are equivalent descriptions of nature. Such relations which equate D3 charge in dual theories are often referred to as \emph{tadpole relations}. While in the physics literature tadpole relations are verified by explicitly computing both sides, there is no reason a priori from a purely mathematical standpoint why such relations should hold. In \cite{TR2} however, it was observed that such identities are basically degree-zero manifestations of a special case of Theorem~\ref{mt}, and thus could be \emph{derived} from solely mathematical principles without having to explicitly compute any Euler characteristics of $Y$ or the $D_i$.  

We now illustrate this phenomenon by considering an explicit weak coupling limit of F-theory first constructed in \cite{EJY}, and then deriving the Chern class identity which encodes the associated tadpole relation in degree zero via Theorem~\ref{mt}. In particular, we consider the elliptic fibration referred to in \cite{EJY} as the \emph{$Q_7$ fibration}, as it is constructed from an elliptic curve whose defining polynomial admits a Newton polygon which is a reflexive quadrilateral with seven lattice points on its boundary. 

So let $B$ be a smooth compact variety over $\mathbb{C}$ endowed with an ample line bundle $\mathscr{L}\to B$, and let $\mathscr{E}\to B$ be the vector bundle given by
\[
\mathscr{E}=\mathscr{O}_B\oplus \mathscr{O}_B\oplus \mathscr{L}.
\] 
Denote by $\pi:\mathbb{P}(\mathscr{E})\to B$ the projective bundle of \emph{lines} in $\mathscr{E}$. We then consider the \emph{$Q_7$ fibration} (which was first constructed in \cite{CCG}), which is a hypersurface in the $\mathbb{P}^2$-bundle $\mathbb{P}(\mathscr{E})$ given by
\[
Y:\left(yx^2-e_1y^3+e_2y^2z+e_3xz^2+e_4yz^2+e_5z^3=0\right)\subset \mathbb{P}(\mathscr{E}),
\]
where $e_i$ is a general section of $\pi^*\mathscr{L}^2$ for $i\neq 3$ and $e_3$ is a general section of $\pi^*\mathscr{L}$. With these prescriptions $Y$ is then the zero-scheme associated with a section of $\mathscr{O}_{\mathbb{P}(\mathscr{E})}(3)\otimes \pi^*\mathscr{L}^2$. Composing the natural inclusion $i:Y\hookrightarrow \mathbb{P}(\mathscr{E})$ with the bundle projection $\pi:\mathbb{P}(\mathscr{E})\to B$ then endows $Y$ with the structure of an elliptic fibration $\varphi=\pi\circ i:Y\to B$. In the case that $B$ is Fano we may take $\mathscr{L}$ to be the anticanonical bundle $\mathscr{O}(-K_B)\to B$ and then a straghtforward calculation using adjunction shows that in such a case $Y$ is in fact Calabi-Yau. If we further restrict the base to be a Fano \emph{threefold} then we are in the context of the physical setting but we require no such assumptions.

A weak coupling limit of F-theory associated with the $Q_7$ fibration was then constructed in \cite{EJY} by deforming the coefficient sections $e_i$ in terms of a deformation parameter $t\in \Delta$ which gives rise to a family $\mathscr{Y}\to \Delta$, in such a way that $\mathscr{Y}$ is given by
\[
\mathscr{Y}:\left(yx^2-\beta y^3+2\vartheta y^2z+t^2\rho xz^2+hyz^2+t\iota z^3=0\right)\subset \mathbb{P}(\mathscr{E})\times \Delta,
\]  
where to emphasize the distinction between $Y$ and $\mathscr{Y}$ we use the notation
\[
e_1=\beta, \quad e_2=2\vartheta, \quad e_3=t^2\rho, \quad e_4=h, \quad e_5=t\iota.
\]
The central fiber is then given by 
\[
Y_0:\left(y(x^2-\beta y^2+2\vartheta yz+hz^2)=0\right)\subset \mathbb{P}(\mathscr{E}).
\]
For every $t\neq 0$ there exists an associated discriminant $\mathscr{D}_t\subset B$ over which the singular fibers of $Y_t$ reside. The flat limit of such discriminants as $t\to 0$ will then be denoted by $\mathscr{D}_0$, and will be referred to as the \emph{limiting discriminant} associated with the family $\mathscr{Y}\to \Delta$. The D3 charge on the type-IIB side is then given by the sum of the Euler characteristics of the components of the limiting discriminant (taken with multiplicities and certain contributions from singularities). In the context at hand the associated limiting discriminant $\mathscr{D}_0$ is given by
\[
\mathscr{D}_0:(h^2\iota^2(\vartheta^2+h\beta)=0)\subset B.
\] 
We then denote the components of $\mathscr{D}_0$ by
\[
O:(h=0), \quad D_1:(\iota^2=0), \quad D_2:(\vartheta^2+h\beta=0).
\]
The notation comes from the fact that physicists refer to $O$ as the `orientifold plane' and the $D_i$ as `D-branes'. We assume that both hypersurfaces given by $h=0$ and $\iota=0$ are smooth and intersect transversally. Now notice that the branes $D_1$ and $D_2$ which arise in the limit admit singularities (as schemes), and in such a case physicists say that the charge associated with the brane is not just $\chi(D_i)$, but $\chi(D_i)-\chi(S_i)$, where the $S_i$ are subvarieties supported on the singular locus of $D_i$ given by
\[
S_1:(\iota=h=0),\quad S_2:(\vartheta=h=\beta=0).
\] 
Note that $S_i$ is the intersection of the singular scheme of $D_i$ with $O$ for both $i=1,2$. In any case, the total D3 chrage $N_{D3}$ on the type-IIB side is given by
\[
N_{D3}=2\chi(O)+2\chi(D_1)-\chi(S_1)+\chi(D_2)-\chi(S_2),
\] 
while the total D3 charge $N_{D3}$ on the F-theory side is given by
\[
N_{D3}=\chi(Y),
\]
so that the tadpole relation predicted by F-theory/type-IIB duality is given by
\begin{equation}\label{trq7}
\chi(Y)\overset{?}=2\chi(O)+2\chi(D_1)-\chi(S_1)+\chi(D_2)-\chi(S_2).
\end{equation}
We now show that \ref{trq7} in fact holds, as we show it is in fact the degree-zero term of
\begin{equation}\label{ciq7}
c_{\text{SM}}\left(\left.{\varphi_t}_*\mathbbm{1}_{\mathscr{Y}}\right|_{Y_t}\right)=c_{\text{SM}}\left({\varphi_0}_*\sigma_F \mathbbm{1}_{\mathscr{Y}}\right)
\end{equation}
as given by Theorem~\ref{mt}, as the family $\mathscr{Y}\to \Delta$ satisfies all the hypotheses of the theorem, with $\pi:\mathbb{P}(\mathscr{E})\to B$ playing the role of $f:X\to B$. For this, first note that
\[
c_{\text{SM}}\left(\left.{\varphi_t}_*\mathbbm{1}_{\mathscr{Y}}\right|_{Y_t}\right)=c_{\text{SM}}(\varphi_*\mathbbm{1}_Y)=\varphi_*c_{\text{SM}}(\mathbbm{1}_Y)=\varphi_*c(Y),
\] 
so that the degree-zero term of the LHS of equation (\ref{ciq7}) indeed coincides with $\chi(Y)$ (since proper pushforwards necessarily preserve terms in degree-zero). As for the RHS of (\ref{ciq7}), we now show that
\begin{equation}\label{rhsci}
c_{\text{SM}}\left({\varphi_0}_*\sigma_F \mathbbm{1}_{\mathscr{Y}}\right)=c_{\text{SM}}\left(2\mathbbm{1}_O+2\mathbbm{1}_{D_1}-\mathbbm{1}_{S_1}+\mathbbm{1}_{D_2}-\mathbbm{1}_{S_2}\right),
\end{equation}
which by Theorem~\ref{mt} yields the Chern class identity
\[
c_{\text{SM}}(\varphi_*\mathbbm{1}_Y)=c_{\text{SM}}\left(2\mathbbm{1}_O+2\mathbbm{1}_{D_1}-\mathbbm{1}_{S_1}+\mathbbm{1}_{D_2}-\mathbbm{1}_{S_2}\right),
\]
from which identity (\ref{trq7}) immediately follows via formula (\ref{ecf}).

To show (\ref{rhsci}), we have to compute
\[
{\varphi_0}_*\sigma_F \mathbbm{1}_{\mathscr{Y}},
\]
thus we now give a precise definition of $\sigma_F \mathbbm{1}_{\mathscr{Y}}$. We use a characterization of $\sigma_F \mathbbm{1}_{\mathscr{Y}}$ given by Aluffi \cite{VSA}, which we state via the following 
\begin{definition}\label{SD} 
Let $\mathcal{Z}\to \mathscr{D}$ be a family over a disk about the origin in $\mathbb{C}$ such that the total space $\mathcal{Z}$ is smooth over $\mathscr{D}\setminus \{0\}$, denote its central fiber by $Z_0$, and let $\psi:\widetilde{\mathcal{Z}}\to \mathcal{Z}$ be a proper birational morphism such that $\widetilde{\mathcal{Z}}$ is smooth, $\mathcal{D}=\psi^{-1}(Z_0)$ is a divisor with normal crossings with smooth components, and $\psi$ restricted to the complement of $\mathcal{D}$ is an isomorphism (such a $\psi$ exists by resolution of singularities). Let $\delta$ be the constructible function on $\mathcal{D}$ given by 
\[
\delta(p)=\begin{cases} m \quad  \text{if $p$ lies on a single component of $\mathcal{D}$ of multiplicity $m$,} \\ 0 \quad \hspace{2mm} \text{otherwise}. \end{cases}
\]   
We then set
\[
\sigma_F \mathbbm{1}_{\mathcal{Z}}={\left.\psi\right|_{\mathcal{D}}}_*\delta,
\]
where ${\left.\psi\right|_{\mathcal{D}}}_*$ denotes the proper pushforward of constrictible functions associated with the restriction of $\psi$ to $\mathcal{D}$ (the definition of proper pushforward appears in (\ref{cfd})).
\end{definition}

We now construct a resolution of singularities $\widetilde{\mathscr{Y}}\to \mathscr{Y}$ satisfying the hypotheses of Definition~\ref{SD} in order to arrive at our associated function $\delta\in F(\widetilde{\mathscr{Y}})$. The singular locus of $\mathscr{Y}$ is the codimension four locus given by
\[
\mathscr{Y}_{\text{sing}}:\left(y=x^2+h=\iota=c=0\right)\subset \mathbb{P}(\mathscr{E})\times \Delta.
\]
As the singularities of $\mathscr{Y}$ are away from $\{z=0\}$, we set $z=1$ and work with a local equation for $\mathscr{Y}$ given by
\[
\mathscr{Y}_{\text{loc}}:\left( yx^2-\beta y^3+2\vartheta y^2+t^2\rho x+hy+t\iota =0\right)\subset \mathbb{A}^2\times B\times \Delta.
\] 
We now blowup $\mathbb{A}^2\times B\times \Delta$ along $\{y=c=0\}$, and work in the chart where the associated pullback is given by
\[
y\mapsto X_1, \quad c\mapsto X_1X_2.
\]
The total transform of $\mathscr{Y}_{\text{loc}}$ is then given by
\[
\widetilde{\mathscr{Y}}_{\text{tot}}:\left(X_1(x^2-\beta X_1^2)+2\vartheta X_1^2+\rho X_1^2X_2^2x+hX_1+\iota X_1X_2=0\right)\subset \widetilde{\mathbb{A}^2\times B\times \Delta},
\]
so that the proper transform of $\mathscr{Y}_{\text{loc}}$ is given by
\[
\widetilde{\mathscr{Y}}_{\text{prop}}:\left(x^2-\beta X_1^2+2\vartheta X_1+\rho X_1X_2^2x+h+\iota X_2=0\right)\subset \widetilde{\mathbb{A}^2\times B\times \Delta}.
\]
It is then straightforward to show that the restriction of the blowup
\[
\widetilde{\mathscr{Y}}_{\text{prop}}\to \mathscr{Y}
\]
is indeed a resolution of singularities satisfying the hypotheses of Definition~\ref{SD}. Now since $c=0$ pulls back to $X_1X_2=0$ under the blowup, the pullback of the cental fiber $Y_0$ is given by
\[
\widetilde{Y_0}:\left(X_1X_2=x^2-\beta X_1^2+2\vartheta X_1+\rho X_1X_2^2x+h+\iota X_2=0\right)\subset \widetilde{\mathbb{A}^2\times B\times \Delta},
\]
which is a divisor with normal crossings with two smooth components given by
\begin{eqnarray*}
\mathcal{D}_1&=&\widetilde{\mathscr{Y}}_{\text{prop}}\cap \{X_1=0\}:(x^2+h+X_2\iota=0)\subset \{X_1=0\} \\
\mathcal{D}_2&=&\widetilde{\mathscr{Y}}_{\text{prop}}\cap \{X_2=0\}:(x^2-\beta X_1^2+2\vartheta X_1 +h=0)\subset \{X_2=0\},  \\
\end{eqnarray*}
which intersect along 
\[
X=\mathcal{D}_1\cap \mathcal{D}_2:\left(x^2+h=0\right)\subset \widetilde{\mathscr{Y}}_{\text{prop}},
\]
which is a smooth double-cover of $B$ ramified over the orientifold plane $O:(h=0)\subset B$. It follows from Definition~\ref{SD} that the constructible function $\delta$ we need to pushforward via the resolution to yield $\sigma_F \mathbbm{1}_{\mathscr{Y}}$ is then given by
\[
\delta=\mathbbm{1}_{\mathcal{D}_1}+\mathbbm{1}_{\mathcal{D}_2}-\mathbbm{1}_{X}.
\]
We then have 
\[
\sigma_F \mathbbm{1}_{\mathscr{Y}}=p_*\delta,
\]
where $p$ denotes the restriction of the resolution to the pullback of the central fiber $\widetilde{Y_0}$. Our goal is then to compute
\[
{\varphi_0}_*\sigma_F \mathbbm{1}_{\mathscr{Y}}={\varphi_0}_*\circ p_*\delta=(\varphi_0\circ p)_*\left(\mathbbm{1}_{\mathcal{D}_1}+\mathbbm{1}_{\mathcal{D}_2}-\mathbbm{1}_{X}\right) \in F(B).
\]
For this, we use the following 
\begin{lemma}\label{l1}
Let $f:Z\to V$ be a proper morphism of varieties and let $\{U_i\}$ be a stratification of $V$ with $U_i$ locally closed such that the fibers are topologically constant on each $U_i$. Denote by $F_i$ the fiber over $U_i$ and write $U_i=V_i\setminus W_i$ with $V_i$ and $W_i$ closed subvarieties of $V$ for each $i$. Then
\[
f_*\mathbbm{1}_{Z}=\sum_i \chi(F_i)\left(\mathbbm{1}_{V_i}-\mathbbm{1}_{W_i}\right).
\] 
\end{lemma}

We omit the proof as it follows directly from the definition of proper pushforward given in (\ref{cfd}). Thus to compute 
\[
(\varphi_0\circ p)_*\left(\mathbbm{1}_{\mathcal{D}_1}+\mathbbm{1}_{\mathcal{D}_2}-\mathbbm{1}_{X}\right),
\]
we view $\mathcal{D}_1$, $\mathcal{D}_2$ and $X$ as fibrations over $B$ and stratify $B$ into strata over which the fibers of the corresponding fibration are topologically constant.

As for $\mathcal{D}_1$, it may be viewed as a conic fibration over $B$, with fibers being smooth conics over $B\setminus D_1$, disjoint $\mathbb{P}^1$s over $D_1\setminus S_1$, and finally two $\mathbb{P}^1$s meeting at a single point over $S_1$, where the varieties $D_1$ and $S_1$ are given by
\[
D_1:\left(\iota^2=0\right), \quad S_1:(\iota=h=0).
\] 
Now since the Euler characteristic of a smooth conic is 2, two disjoint $\mathbb{P}^1$s is 4, and two $\mathbb{P}^1$s meeting at a point is 3, it follows via Lemma~\ref{l1} that the pushforward of $\mathbbm{1}_{\mathcal{D}_1}$ is given by
\[
\mathbbm{1}_{\mathcal{D}_1}\mapsto 2(\mathbbm{1}_B-\mathbbm{1}_{D_1})+4(\mathbbm{1}_{D_1}-\mathbbm{1}_{S_1})+3\mathbbm{1}_{S_1}=2\mathbbm{1}_B+2\mathbbm{1}_{D_1}-\mathbbm{1}_{S_1}.
\]

As for $\mathcal{D}_2$, it is also a conic fibration over $B$, with fibers being smooth conics over $B\setminus D_2$, two $\mathbb{P}^1$s intersecting at a point over $D_2\setminus S_2$, and a double line over $S_2$, where the varieties $D_2$ and $S_2$ are given by
\[
D_2:(\vartheta^2+\beta h), \quad S_2:(\vartheta=\beta=h=0).
\]
It then follows by computing Euler characteristics of the fibers that the pushforward of $\mathbbm{1}_{\mathcal{D}_2}$ is given by
\[
\mathbbm{1}_{\mathcal{D}_2}\mapsto 2(\mathbbm{1}_B-\mathbbm{1}_{D_2})+3(\mathbbm{1}_{D_2}-\mathbbm{1}_{S_2})+2\mathbbm{1}_{S_1}=2\mathbbm{1}_B+\mathbbm{1}_{D_2}-\mathbbm{1}_{S_2}.
\]

Now finally since $X\to B$ is a double cover ramified over $O:(h=0)\subset B$ we have that the pushforward of $\mathbbm{1}_X$ is given by
\[
\mathbbm{1}_X\mapsto 2\mathbbm{1}_B-\mathbbm{1}_O.
\]
Putting things all together we get
\[
{\varphi_0}_*\sigma_F \mathbbm{1}_{\mathscr{Y}}=2\mathbbm{1}_O+2\mathbbm{1}_{D_1}-\mathbbm{1}_{S_1}+\mathbbm{1}_{D_2}-\mathbbm{1}_{S_2},
\]
thus by Theorem~\ref{mt} we have
\[
c_{\text{SM}}(\varphi_*\mathbbm{1}_Y)=c_{\text{SM}}\left(2\mathbbm{1}_O+2\mathbbm{1}_{D_1}-\mathbbm{1}_{S_1}+\mathbbm{1}_{D_2}-\mathbbm{1}_{S_2}\right),
\]
as desired. We note that this identity which encodes the tadpole relations associated with the weak coupling limit in degree-zero hold over a base $B$ of arbitrary dimension, and without any Calabi-Yau hypothesis on $Y$. We then close with two remarks.

\begin{remark}
We would like to point out a couple of apparent differences between the form of our tadpole relations (and associated Chern class identities) and similar relations derived in the physics literature. In particular, in \cite{EJY}, their presumed tadpole relation takes the form
\[
2\chi(Y)=4\chi(O)+2\chi(\overline{D_1})+\chi(\overline{D}_2)-\chi(\overline{S_2}),
\]
where the overline over a variety denotes its pullback to the double cover $X\to B$, where we recall $X$ is given by
\[
X:(x^2+h=0)\subset \widetilde{Y}_0.
\]
The difference in appearance from the tadpole relation given here, namely
\[
\chi(Y)=2\chi(O)+2\chi(D_1)-\chi(S_1)+\chi(D_2)-\chi(S_2),
\]
comes from the fact that we prefer to state the tadpole relations in terms of subvarieties of the base $B$, rather than subvarieties of $X$, which the physicists prefer to do. Working in the double cover then produces certain factors of two to appear in the physicist's identities which don't appear ours. Another difference in our approach we would like to point out is that while in the physics literature the tadpole relations and associated Chern class identities are first guessed and then verified by explicit computation, Theorem~\ref{mt} yields the identities from first principles, without computing any Chern classes. 
\end{remark}

\begin{remark}
One curious aspect of D3 charge on the type-IIB side is that when an irreducible component $D_i$ of the limiting discriminant different from the orientifold $O$ is singular, the contribution of $D_i$ to the D3 charge is not $\chi(D_i)$ but rather $\chi(D_i)-\chi(S_i)$, where $S_i$ is the singular \emph{scheme} of $D_i$ intersected with $O$. By singular scheme we mean the subscheme of $D_i$ corresponding to the ideal generated by its defining equation and its partial derivatives. For example $D_1$ is given by
\[
D_1:(\iota^2=0)\subset B,
\]
thus its singular scheme is given by 
\[
{D_1}_{\text{sing}}:(\iota=0)\subset B,
\]
so that $S_1:(\iota=h=0)\subset B$ is precisely ${D_1}_{\text{sing}}\cap O$. Ad hoc explanations for negative contributions to the D3 charge coming from singularities were given in the physics literature for the case of $D_2$ \cite{CDM}, but in all examples known to our knowledge, the negative contribution to the D3 charge for a singular brane is precisely the Euler characteristic of its singular scheme intersected with $O$, and moreover this may be incorperated into a general definition of `orientifold Euler characteristic'. In particular, if $D$ is an irreducible component of a limiting discriminant different from the orinetifold plane corresponding to a weak coupling limit of F-theory, we propose to define the orientifold Euler characteristic of $D$ to be
\[
\chi_o(D)=m\chi(D)-\chi(S\cap O),
\]
where $m$ is the multiplicity of $D$ and $S$ denotes its singular scheme. With this definition all known examples of tadpole relations may be written as
\[
\chi(Y)=2\chi(O)+\sum_i \chi_o(D_i).
\]
In any case, Theorem~\ref{mt} yields the proper negative contribution to the D3 charge coming from the singularities of branes without any need for further explanation from the mathematical viewpoint.    
\end{remark}

\noindent\emph{Acknowledgements.} The first author would like to thank the organizers of the conference ``International Conference on Singularity Theory, in Honor of Henry Laufer's 70th Birthday'', which took place in December 2015 at the Tsinghua Sanya International Mathematics Forum in Sanya China, where he got to speak on the subject of this paper. The second author's research is supported by Tianyuan Grant No. 11526046.

\bibliographystyle{plain}
\bibliography{RVS2}

\begin{thebibliography}{10}

\bibitem{VSA}
Paolo Aluffi.
\newblock Verdier specialization via weak factorization.
\newblock {\em Ark. Mat.}, 51(1):1--28, 2013.

\bibitem{AE1}
Paolo Aluffi and Mboyo Esole.
\newblock Chern class identities from tadpole matching in type {IIB} and
  {F}-theory.
\newblock {\em J. High Energy Phys.}, (3):032, 29, 2009.

\bibitem{AE2}
Paolo Aluffi and Mboyo Esole.
\newblock New orientifold weak coupling limits in {F}-theory.
\newblock {\em J. High Energy Phys.}, (2):020, i, 52, 2010.

\bibitem{BSCC}
J.-P. Brasselet and M.-H. Schwartz.
\newblock Sur les classes de {C}hern d'un ensemble analytique complexe.
\newblock In {\em The {E}uler-{P}oincar\'e characteristic ({F}rench)},
  volume~82 of {\em Ast\'erisque}, pages 93--147. Soc. Math. France, Paris,
  1981.

\bibitem{CCG}
Sergio~L. Cacciatori, Andrea Cattaneo, and Bert van Geemen.
\newblock A new {CY} elliptic fibration and tadpole cancellation.
\newblock {\em J. High Energy Phys.}, (10):031, 20, 2011.

\bibitem{CDM}
Andr{\'e}s Collinucci, Frederik Denef, and Mboyo Esole.
\newblock D-brane deconstructions in {IIB} orientifolds.
\newblock {\em J. High Energy Phys.}, (2):005, 57, 2009.

\bibitem{EFY}
Mboyo Esole, James Fullwood, and Shing-Tung Yau.
\newblock {$D_5$} elliptic fibrations: non-{K}odaira fibers and new orientifold
  limits of {F}-theory.
\newblock {\em Commun. Number Theory Phys.}, 9(3):583--642, 2015.

\bibitem{EJY}
Mboyo Esole, Monica Jinwoo~Kang, and S.-T. Yau.
\newblock A {N}ew {M}odel for {E}lliptic {F}ibrations with a {R}ank {O}ne
  {M}ordell-{W}eil {G}roup: I. {S}ingular {F}ibers and {S}emi-{S}table
  {D}egenerations.
\newblock {\em arXiv:1410.0003}, 2014.

\bibitem{TR2}
J.~Fullwood.
\newblock On tadpole relations via {V}erdier specialization.
\newblock {\em Journal of Geometry and Physics}, (104C):54--63, 2016.

\bibitem{IT}
William Fulton.
\newblock {\em Intersection theory}, volume~2 of {\em Ergebnisse der Mathematik
  und ihrer Grenzgebiete. 3. Folge.}
\newblock Springer-Verlag, Berlin, second edition, 1998.

\bibitem{RMCC}
R.~D. MacPherson.
\newblock Chern classes for singular algebraic varieties.
\newblock {\em Ann. of Math. (2)}, 100:423--432, 1974.

\bibitem{PPCC}
Adam Parusi{\'n}ski and Piotr Pragacz.
\newblock Characteristic classes of hypersurfaces and characteristic cycles.
\newblock {\em J. Algebraic Geom.}, 10(1):63--79, 2001.

\bibitem{MHSCC}
Marie-H{\'e}l{\`e}ne Schwartz.
\newblock Classes caract\'eristiques d\'efinies par une stratification d'une
  vari\'et\'e analytique complexe. {I}.
\newblock {\em C. R. Acad. Sci. Paris}, 260:3262--3264, 1965.

\bibitem{VS}
J.-L. Verdier.
\newblock Sp\'ecialisation des classes de {C}hern.
\newblock In {\em The {E}uler-{P}oincar\'e characteristic ({F}rench)},
  volume~82 of {\em Ast\'erisque}, pages 149--159. Soc. Math. France, Paris,
  1981.

\end{thebibliography}

\end{document}